\documentclass[submission]{FPSAC2019}

\usepackage[capitalize,noabbrev]{cleveref}
\usepackage[backgroundcolor=yellow]{todonotes}
\usepackage{tikz}
\usepackage{tikz-cd}
\usepackage{mathscinet}
\usepackage{enumitem}

\usepackage{CJKutf8}

\newcommand{\Z}{\mathbb{Z}}

\newcommand{\Q}{\mathbb{C}}

\newcommand{\Sym}[1]{\mathfrak{S}_{#1}}
\newcommand{\Spin}[1]{\widetilde{\mathfrak{S}}_{#1}}

\newcommand{\p}{\Ch^{\mathrm{spin}}}
\newcommand{\Ch}{\mathrm{Ch}}
\newcommand{\ChD}{\widetilde{\Ch}}

\newcommand{\class}{\pi}
\newcommand{\partition}{I}

\newcommand{\CHIL}[2]{\chi^{#1}\left(#2\right)}

\newcommand{\CHIS}[2]{X^{#1}\left(#2\right)}

\newcommand{\CHIT}[2]{\widetilde{\phi}^{#1}\left(#2\right)}

\newcommand{\proP}{\mathbf{p}}
\newcommand{\proQ}{\mathbf{q}}

\DeclareRobustCommand{\stirlingS}{\genfrac\{\}{0pt}{}}
\DeclareRobustCommand{\stirlingF}{\genfrac[]{0pt}{}}

\DeclareMathOperator{\SP}{SP}
\DeclareMathOperator{\OP}{OP}

\theoremstyle{definition}
\newtheorem{definition}{Definition}[section]

\theoremstyle{theorem}
\newtheorem{theorem}[definition]{Theorem}

\newtheorem{corollary}[definition]{Corollary}

\theoremstyle{remark}
\newtheorem{example}[definition]{Example}

\title[Spin Stanley character formula]{Stanley character formula \\ for the spin characters of the symmetric groups}

\author[%
Sho Matsumoto, 
Piotr \'{S}niady]%
{Sho Matsumoto%
\thanks{\href{mailto:{shom@sci.kagoshima-u.ac.jp}}{{shom@sci.kagoshima-u.ac.jp}}.
Research of SM was supported by JSPS KAKENHI Grant Number 17K05281.}%
\addressmark{1}, 
\and 
Piotr \'{S}niady%
\thanks{\href{mailto:psniady@impan.pl}{psniady@impan.pl}. 
Research of P\'{S} was supported by \emph{Narodowe Centrum Nauki}, grant number
2017/26/A/ST1/00189.}\addressmark{2}}

\address{%
\addressmark{1}Graduate School of Science and Engineering, Kagoshima University 
1-21-35, Korimoto, Kagoshima, Japan	 \\
\addressmark{2}Institute of Mathematics, Polish Academy of Sciences,
	ul.~\'Sniadeckich 8, 00-956 Warszawa, Poland 
}

\received{November 1, 2018}

\abstract{%
	We give a new formula for the irreducible spin characters of the symmetric
	groups. This formula is analogous to Stanley's character formula for the usual
	(linear) characters of the symmetric groups. }

\resumetitle{Resum\'e} 

\resume{Nous donnons une nouvelle formule pour les caractères de spin irréductibles
des groupes symétriques. Cette formule est analogue à la formule de Stanley pour
les caractères habituels (linéaires) des groupes symétriques.}


\begin{document}

\maketitle

A full version (which contains complete versions of all proofs) of this extended
abstract will be available soon \cite{SpinAlgebraic2018} and will be published
elsewhere.
	
\bigskip

The \emph{spin symmetric group} $\Spin{n}$ is the double cover of
the symmetric group $\Sym{n}$. This group is generated by $t_1,\dots,t_{n-1},z$
subject to the relations:
\begin{align*}
z^2   &= 1 ,\\
z t_i &= t_i z, & t_i^2&= z & \text{for $i\in [n-1]$}, \\
(t_i t_{i+1})^3 &= z  & & & \text{for $i\in [n-2]$}, \\
t_i t_j &=  z t_j t_i & & & \text{for $|i-j|\geq 2$};
\end{align*}	
we use the convention that $[k]=\{1,\dots,k\}$. This
group was introduced by Schur \cite{Schur1911}; it is essential for studying
\emph{projective representations} of the usual symmetric group $\Sym{n}$.

\medskip

Schur proved that, roughly speaking, the conjugacy classes of $\Spin{n}$ which
are non-trivial from the viewpoint of the character theory are indexed by
\emph{odd partitions} of $n$, i.e.~partitions  $(\lambda_1,\dots,\lambda_l)$ of
$n$ such that $\lambda_1\geq \cdots\geq \lambda_l$ are odd positive integers.
The set of such odd partitions of $n$ will be denoted by $\OP_n$.

The central element $z\in\Spin{n}$ acts on each irreducible representation by
$\pm 1$. An irreducible representation of $\Spin{n}$ is said to be \emph{spin}
if $z$ corresponds to $-1$. Schur also proved that the irreducible spin
representations of $\Spin{n}$, roughly speaking, correspond to \emph{strict
	partitions} of $n$, i.e.~to partitions  $(\lambda_1,\dots,\lambda_l)$ of $n$
which form a \emph{strictly} decreasing sequence $\lambda_1>\cdots>\lambda_l$ of
positive integers. The set of such strict partitions of $n$ will be denoted by
$\SP_n$. We will represent them as
\emph{shifted Young diagrams}, cf.~\cref{fig:double}.

\begin{figure}[t]
	\centerline{
		\begin{tikzpicture}[xscale=0.6,yscale=0.6]
		\begin{scope}
		\clip (0,0) -- (0,1) -- (1,1) -- (1,2) -- (2,2) -- (2,3) -- (3,3) -- (4,3) -- (4,2) -- (6,2) -- (6,1) -- (6,0);
		\draw[gray] (0,0) grid (8,3);
		\end{scope}
		\draw[->] (0,0) -- (10,0) node[anchor=west]{$x$};
		\draw[->] (0,0) -- (0,4) node[anchor=south]{$y$};	   
		\draw[ultra thick] (0,0) -- (0,1) -- (1,1) -- (1,2) -- (2,2) -- (2,3) -- (3,3) -- (4,3) -- (4,2) -- (6,2) -- (6,1) -- (6,0) -- cycle;
		\draw[line width=0.2cm, opacity=0.4,blue,line cap=round] (0.5,0.5) -- (5.5,0.5);
		\draw[line width=0.2cm, opacity=0.4,blue,line cap=round] (1.5,1.5) -- (5.5,1.5);
		\draw[line width=0.2cm, opacity=0.4,blue,line cap=round] (2.5,2.5) -- (3.5,2.5);	   
		\end{tikzpicture}
		\qquad
		\begin{tikzpicture}[xscale=0.6,yscale=0.6]
		\begin{scope}[xshift=1cm]
		\clip (0,0) -- (0,1) -- (1,1) -- (1,2) -- (2,2) -- (2,3) -- (3,3) -- (4,3) -- (4,2) -- (6,2) -- (6,1) -- (6,0);
		\draw[gray] (0,0) grid (8,3);
		\end{scope}
		\begin{scope}[rotate=90,yscale=-1]
		\clip (0,0) -- (0,1) -- (1,1) -- (1,2) -- (2,2) -- (2,3) -- (3,3) -- (4,3) -- (4,2) -- (6,2) -- (6,1) -- (6,0);
		\draw[gray] (0,0) grid (8,3);
		\end{scope}    
		\begin{scope}[xshift=1cm]
		\draw[ultra thick] (0,0) -- (0,1) -- (1,1) -- (1,2) -- (2,2) -- (2,3) -- (3,3) -- (4,3) -- (4,2) -- (6,2) -- (6,1) -- (6,0) -- cycle;
		\draw[line width=0.2cm, opacity=0.4,blue,line cap=round] (0.5,0.5) -- (5.5,0.5);
		\draw[line width=0.2cm, opacity=0.4,blue,line cap=round] (1.5,1.5) -- (5.5,1.5);
		\draw[line width=0.2cm, opacity=0.4,blue,line cap=round] (2.5,2.5) -- (3.5,2.5);	 
		\end{scope}
		\begin{scope}[rotate=90,yscale=-1]
		\draw[ultra thick] (0,0) -- (0,1) -- (1,1) -- (1,2) -- (2,2) -- (2,3) -- (3,3) -- (4,3) -- (4,2) -- (6,2) -- (6,1) -- (6,0) -- cycle;
		\draw[line width=0.2cm, opacity=0.4,OliveGreen,line cap=round] (0.5,0.5) -- (5.5,0.5);
		\draw[line width=0.2cm, opacity=0.4,OliveGreen,line cap=round] (1.5,1.5) -- (5.5,1.5);
		\draw[line width=0.2cm, opacity=0.4,OliveGreen,line cap=round] (2.5,2.5) -- (3.5,2.5);	
		\end{scope}    
		\draw[line width=3.2pt, dashed,red] (0,0) -- (1,0) -- (1,1) -- (2,1) -- (2,2) -- (3,2) -- (3,3) -- (4,3) -- (4,4) -- (5,4) -- (5,5) -- (6,5) -- (6,6);
		\end{tikzpicture}
	}
	
	\caption{ Strict partition $\lambda=(6,5,2)$ shown as a \emph{shifted Young
		diagram} and its double $D(\lambda)=(7,7,5,3,2,2)$. } 
\label{fig:double}
\end{figure}
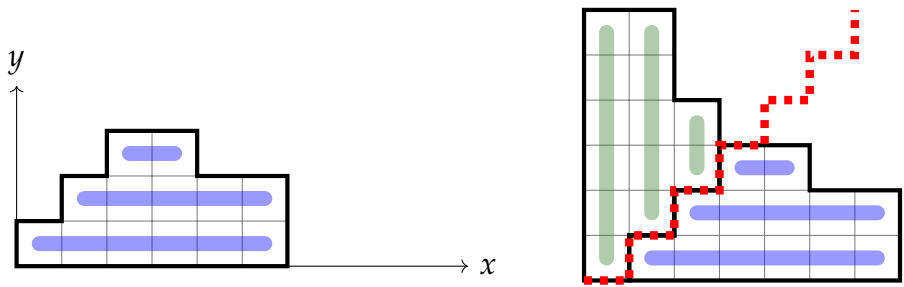

\medskip

For an odd partition $\pi\in\OP_n$ (which corresponds to a conjugacy class of
$\Spin{n}$) and a strict partition $\lambda\in\SP_n$ (which corresponds to its
irreducible spin representation) we denote by $\CHIT{\lambda}{\pi}$ the
corresponding \emph{spin character}
(for some fine details related to this definition we refer to \cite[Section 2]{Ivanov2004}).
\textbf{Our goal is to give a closed formula for such spin characters which
	would be useful for the purposes of the asymptotic representation theory,
	i.e.~which would allow good understanding of the limit $\lambda\to\infty$.} %

\medskip

In the following it will be more convenient to pass to quantities
\[\CHIS{\lambda}{\pi}:=2^{\frac{\ell(\lambda)-\ell(\pi)}{2}}\ \CHIT{\lambda}{\pi},\]
where $\ell(\pi)$ denotes the number of parts of a partition $\pi$,
cf.~\cite[Proposition 3.3]{Ivanov2004}.

\section{Normalized characters}

The usual way of viewing the linear characters $\CHIL{\lambda}{\pi}$ of the
symmetric group $\Sym{n}$ is to fix the irreducible representation $\lambda$ and
to consider the character as a function of the conjugacy class $\class$. 
The \emph{dual approach}, initiated by Kerov and Olshanski \cite{Kerov1994},
suggests to do the opposite: \emph{fix the conjugacy class $\class$ and to view
	the character as a function of the irreducible representation $\lambda$}. 
In
order for this approach to be successful one has to choose the most convenient
normalization constants which we review in the following.
\medskip
	
Following Kerov and Olshanski \cite{Kerov1994}, 
for a fixed integer partition $\class$ the
corresponding \emph{normalized linear character on the conjugacy class $\class$}
is the function on the set of \emph{all} Young diagrams given by
	\[ \Ch_\class(\lambda):=\begin{cases}
	n^{\downarrow k} 
	\ \frac{ \CHIL{\lambda}{\class\cup 1^{n-k}} }{ \CHIL{\lambda}{1^{n}} } & 
	\text{if } n\geq k, \\
	0 & \text{otherwise,}
	\end{cases}\]
	where $n=|\lambda|$ and $k=|\class|$ and $n^{\downarrow k}=n (n-1) \cdots
(n-k+1)$ denotes the falling power. Above, for partitions $\lambda,\sigma\vdash
n$ we denote by $\CHIL{\lambda}{\sigma}$ the irreducible linear character of the
symmetric group which corresponds to the Young diagram $\lambda$, evaluated on
any permutation with the cycle decomposition given by $\sigma$.

	\medskip
	
	Following Ivanov \cite{Ivanov2004,Ivanov2006}, for a fixed odd partition
$\class\in\OP$ the corresponding \emph{normalized spin character} is a function
on the set of \emph{all} strict partitions given by
	\begin{equation}
	\label{eq:projective-normalized}
	\p_\class(\lambda):=\begin{cases}
	n^{\downarrow k}\ 
	\ \frac{ \CHIS{\lambda}{\class\cup 1^{n-k}}}{\CHIS{\lambda}{1^{n}}} = 
	n^{\downarrow k}\ 2^{\frac{k-\ell(\class)}{2}}
		\ \frac{ \CHIT{\lambda}{\class\cup 1^{n-k}}}{ \CHIT{\lambda}{1^{n}}}
	 & \text{if } n\geq k, \\
	0 & \text{otherwise,}
	\end{cases}
	\end{equation}
	where $n=|\lambda|$,  $k=|\class|$, and $\ell(\pi)$ denotes the number of parts of $\pi$.
\textbf{%
	We will find a closed formula for such spin characters $\p_\pi$. We will
	achieve it by finding a link between the families $(\p_\pi)$ and $(\Ch_\pi)$ of
	spin and linear characters.}

	\section{Stanley character formulas}

Let $\sigma_1,\sigma_2\in\Sym{k}$ be permutations and let
$\lambda$ be a Young diagram. Following \cite{FeraySniady2011a}, we say that
\emph{$(f_1,f_2)$ is a coloring of $(\sigma_1,\sigma_2)$ which is compatible
	with $\lambda$} if:
\begin{itemize}
	\item $f_i\colon C(\sigma_i) \to \Z_+$ is a function on the set of cycles of
	$\sigma_i$ for each $i\in\{1,2\}$; we view the values of $f_1$ as columns of
	$\lambda$ and the values of $f_2$ as rows;
	
	\item whenever $c_1\in C(\sigma_1)$ and $c_2\in C(\sigma_2)$ are cycles which are not
	disjoint, the box with the Cartesian coordinates
	$\big( f_1(c_1), f_2(c_2) \big)$ belongs to $\lambda$.
\end{itemize} 
We denote by $N_{\sigma_1,\sigma_2}(\lambda)$ the number of  colorings of
$(\sigma_1,\sigma_2)$ which are compatible with $\lambda$.

\begin{figure}
	\centerline{
		\begin{tikzpicture}[scale=1]
		\draw [->] (0,0) -- (5,0) node[anchor=west] {$x$};
		\draw [->] (0,0) -- (0,3) node[anchor=west] {$y$};
		\begin{scope}
		\clip (0,0) -- (3,0) -- (3,1) -- (1,1) -- (1,2) -- (0,2);
		\draw (0,0) grid (2,2);
		\end{scope}
		\draw[ultra thick] (0,0) -- (3,0) -- (3,1) -- (1,1) -- (1,2) -- (0,2) -- cycle;
		\draw[line width=7pt, opacity=0.3, draw=red] (0.5,-0.5) node[anchor=north,opacity=1, text=black]{$f_1(V)$} -- (0.5,2.5); 
		\draw[line width=7pt, opacity=0.3, draw=blue] (2.5,-0.5) node[anchor=north,opacity=1, text=black]{$f_1(W)$} -- (2.5,1.5); 
		\draw[line width=7pt, opacity=0.3, draw=OliveGreen] (-0.5,0.5) node[anchor=east,opacity=1, text=black]{$f_2(\Pi)$} -- (3.5,0.5); 
		\draw[line width=7pt, opacity=0.3, draw=black] (-0.5,1.5) node[anchor=east,opacity=1, text=black]{$f_2(\Sigma)$} -- (1.5,1.5); 
		\end{tikzpicture}
	} 
	\caption{Graphical representation of the coloring \eqref{eq:coloring} of the
		permutations \eqref{eq:example}  which is compatible with the Young diagram
		$\lambda=(3,1)$.} \label{fig:embed}
\end{figure}
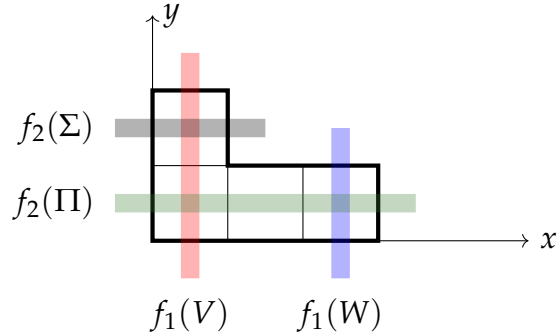

\begin{example}
Let
\begin{equation}
\label{eq:example} 
\sigma_1=\underbrace{(1,5,4,2)}_{\textcolor{red}{V}}\underbrace{(3)}_{\textcolor{blue}{W}},\qquad  \sigma_2 = \underbrace{(2,3,5)}_{\textcolor{OliveGreen}{\Pi}}\underbrace{(1,4)}_{\Sigma}.
\end{equation}
There are three pairs of cycles $(\sigma_1,\sigma_2)\in C(\sigma_1)\times
C(\sigma_2)$ with the property that $\sigma_1$ and $\sigma_2$ are not disjoint,
namely $(V,\Pi), (V,\Sigma), (W,\Pi)$.
It is now easy to check graphically (cf.~\cref{fig:embed}) that $(f_1,f_2)$ is
indeed a coloring compatible with $\lambda=(3,1)$ for
\begin{equation} 
\label{eq:coloring}
f_1(V)=1, \quad f_1(W)=3, \quad f_2(\Pi)=1, \quad f_2(\Sigma)=2.
\end{equation}

By considering four possible choices for the values of $f_2$ and counting the
choices for the values of $f_1$ one verify that $N_{\sigma_1,\sigma_2}(\lambda)=
3^2 +  3 + 1 + 1 =14$ for $\lambda=(3,1)$.
\end{example}

\subsection{Linear Stanley character formula} 

Stanley \cite{Stanley2006} conjectured a certain closed formula for the  linear
characters of the symmetric groups. One of its proofs \cite{FeraySniady2011a}
was obtained by rewriting it in an equivalent form which we will recall in the
following. We refer to \cite{Sniady2016b} for more context.

We will identify a given integer partition $\pi=(\pi_1,\dots,\pi_\ell)\vdash k$
with an arbitrary permutation $\pi\in\Sym{k}$ with the corresponding cycle
structure. For example, we may take
\[ \pi=(1,2,\dots,\pi_1)(\pi_1+1,\pi_1+2,\dots,\pi_1+\pi_2) \cdots\in\Sym{k}.\]

\begin{theorem}[\cite{FeraySniady2011a}]
	\label{thm:stanley-linear} For any partition $\pi\vdash k$ and any Young diagram
$\lambda$
	\begin{equation}  
	\label{eq:stanley-linear} 
	\Ch_\pi(\lambda) = \sum_{\substack{\sigma_1,\sigma_2\in\Sym{k} \\ \sigma_1 \sigma_2=\pi }}  (-1)^{\sigma_1} N_{\sigma_1,\sigma_2} (\lambda), 
	\end{equation}
where $(-1)^{\sigma_1}\in\{-1,1\}$ denotes the sign of the permutation $\sigma_1$.
\end{theorem}

This linear Stanley character formula is closely related to \emph{Kerov
	polynomials} \cite{Stanley2006}, which are expressions of the characters $\Ch_\pi$ in terms of
\emph{free cumulants} of Young diagrams. Recently, the first author
\cite{Matsumoto2018} found spin counterparts for Kerov polynomials. The current
paper was initiated by attempts to understand the underlying structures behind this result.

\subsection{The main result: spin Stanley character formula}

For a strict partition $\lambda\in\SP_n$ we consider its \emph{double}
$D(\lambda)$ which is an integer partition of $2n$. Graphically, $D(\lambda)$
corresponds to a Young diagram obtained by arranging the \emph{shifted Young
	diagram} $\lambda$ and its `transpose' so that they nicely fit along the
`diagonal', cf.~\cref{fig:double}, see also \cite[page 9]{Macdonald1995}.

For $\sigma_1,\sigma_2\in\Sym{k}$ we denote by $|\sigma_1\vee \sigma_2|$ the
number of orbits of the group $\langle \sigma_1,\sigma_2\rangle$ generated by
$\sigma_1$ and $\sigma_2$.
As before,
we identify an integer partition $\pi\vdash k$ with an arbitrary permutation
$\pi\in\Sym{k}$ with the corresponding cycle structure.

\begin{theorem}[The main result]
	\label{thm:spin-Stanley}
	For any odd partition $\pi\in\OP_k$ and $\lambda\in\SP$
	\begin{equation}
	\label{eq:spin-stanley}
	\p_\pi(\lambda) = \sum_{\substack{\sigma_1,\sigma_2\in\Sym{k} \\ \sigma_1 \sigma_2=\pi }} \frac{1}{2^{|\sigma_1\vee \sigma_2|}}\ (-1)^{\sigma_1}\ N_{\sigma_1,\sigma_2}\big( D(\lambda) \big). 
	\end{equation}
\end{theorem}

The remaining sections of this paper
(\cref{sec:linear-in-terms-of-spin,sec:spin-linear,sec:proof}) are devoted to a
sketch of the proof of this result. In the following we will discuss some of its
applications.

\subsection{Application: bounds on spin characters}

The following character bound is a spin version of an analogous result for the
linear characters of the symmetric group \cite{FeraySniady2011a}. It is a direct
application of \cref{thm:spin-Stanley} and its proof follows the same line as
its linear counterpart \cite{FeraySniady2011a}.

\begin{corollary}
\label{cor:Annals}
There exists a universal constant $a>0$ with the property that for any integer
$n\geq 1$, any strict partition $\lambda\in\SP_n$, and any odd partition
$\pi\in\OP_n$
\[
2^{\frac{n-\ell(\pi)}{2}} \ 
\left| \frac{\CHIT{\lambda}{\pi}}{\CHIT{\lambda}{1^n}} \right| 
=
\left| \frac{\CHIS{\lambda}{\pi}}{\CHIS{\lambda}{1^n}} \right| < \left[a \max\left( \frac{\lambda_1}{n} , \frac{n-\ell(\pi)}{n} \right) \right]^{n-\ell(\pi)}. \]
\end{corollary}	

Several asymptotic results about (random) Young diagrams and tableaux which use
the inequality from \cite{FeraySniady2011a} can be generalized in a rather
straightforward way to (random) \emph{shifted} Young diagrams and \emph{shifted}
tableaux thanks to \cref{cor:Annals}. A good example is provided by the results
about the asymptotics of the number of skew standard Young tableaux of
prescribed shape \cite{Dousse2017} which can be generalized in this way to
asymptotics of the number of skew \emph{shifted} standard Young tableaux.

\subsection{Application: characters on mutltirectangular Young diagrams}

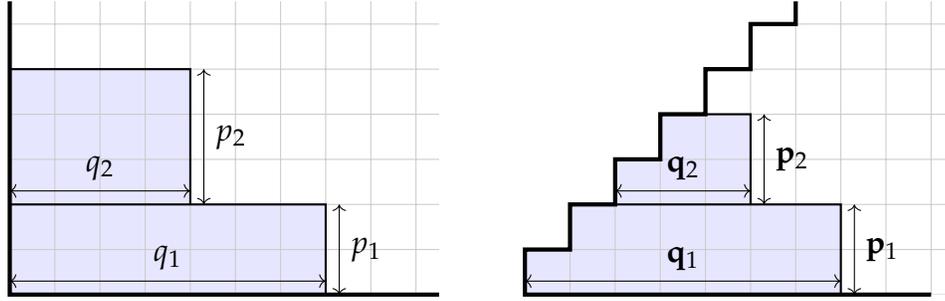
\begin{figure}
	\centerline{
			\begin{tikzpicture}[scale=0.6]	   
		\clip (-0.3,-0.3) rectangle (9.5,6.5);
		\fill[blue!10] (0,0) -- (7,0) -- (7,2) -- (4,2) -- (4,5) -- (0,5);
		\draw[black!20] (0,0) grid (10,7);
		\draw[ultra thick] (10,0) -- (0,0) -- (0,7);
		\draw[thick] (7,0) -- (7,2) -- (0,2);
		\draw[<->] (7.3,0) -- (7.3,2) node[midway,right] {$p_1$};
		\draw[<->] (0,0.3) -- (7,0.3) node[midway,above] {$q_1$};
		\draw[thick] (4,2) -- (4,5) -- (0,5);
		\draw[<->] (4.3,2) -- (4.3,5) node[midway,right] {$p_2$};
		\draw[<->] (0,2.3) -- (4,2.3) node[midway,above] {$q_2$};
		\end{tikzpicture}
		\qquad
	\begin{tikzpicture}[scale=0.6]	 
	\clip (0.7,-0.3) rectangle (10.5,6.5);
	\fill[blue!10] (4,4) -- (6,4) -- (6,2)-- (8,2) -- (8,0) -- (1,0) -- (1,1) -- (2,1) -- (2,2) -- (3,2) -- (3,3) -- (4,3) -- (4,4) -- (5,4);  
	\begin{scope}
	\clip (11,0) -- (1,0) -- (1,1) -- (2,1) -- (2,2) -- (3,2) -- (3,3) -- (4,3) -- (4,4) -- (5,4) -- (5,5) -- (6,5) -- (6,6) -- (7,6) -- (7,7) -- (11,7);
	\draw[black!20] (0,0) grid (11,11);
	\end{scope}
	\draw[ultra thick] (10,0) -- (1,0) -- (1,1) -- (2,1) -- (2,2) -- (3,2) -- (3,3) -- (4,3) -- (4,4) -- (5,4) -- (5,5) -- (6,5) -- (6,6) -- (7,6) -- (7,7);
	\draw[thick] (4,4) -- (6,4) -- (6,2) ;
	\draw[thick] (2,2) -- (8,2) -- (8,0);
	\draw[<->] (8.3,0) -- (8.3,2) node[midway,right] {$\proP_1$};
	\draw[<->] (1,0.3) -- (8,0.3) node[midway,above] {$\proQ_1$};
	\draw[<->] (6.3,2) -- (6.3,4) node[midway,right] {$\proP_2$};
	\draw[<->] (3,2.3) -- (6,2.3) node[midway,above] {$\proQ_2$};
	\end{tikzpicture}
} 
\caption{Multirectangular Young diagram $P\times Q$ and multirectangular
shifted Young diagram $\mathbf{P}\rtimes \mathbf{Q}$.} 
\label{fig:multirectangular-strict}
\end{figure}

Following Stanley \cite{Stanley2003/04}, for tuples of integers
$P=(p_1,\dots,p_l)$, $Q=(q_1,\dots,q_l)$ which fulfil some obvious inequalities
we consider the corresponding \emph{multirectangular Young diagram} $P\times Q$,
cf.~\cref{fig:multirectangular-strict}. Stanley
\cite{Stanley2003/04,Stanley2006} initiated investigation of the characters
$\Ch_\pi(P\times Q)$ viewed as polynomials in the multirectangular coordinates
$p_1,\dots,p_l,q_1,\dots,q_l$; these polynomials now are referred to as
\emph{(linear) Stanley character polynomials}.

The number of colorings $N_{\sigma_1,\sigma_2}(P\times
Q)\in\Z[p_1,\dots,p_l,q_1,\dots,q_l]$ is given by a very explicit, convenient
polynomial. In this way the linear Stanley formula (\cref{thm:stanley-linear})
gives an explicit expression for the linear Stanley polynomials.

\medskip

De Stavola \cite{DeStavolaThesis} adapted these concepts to shifted
multirectangular Young diagrams $\mathbf{P}\rtimes \mathbf{Q}$
cf.~\cref{fig:multirectangular-strict} and initiated investigation of \emph{spin
	Stanley polynomials} $\p_\pi(\mathbf{P}\rtimes \mathbf{Q})$. Thanks to
\cref{thm:spin-Stanley}, by expressing the multirectangular coordinates $P,Q$ of
the double $P\times Q=D(\mathbf{P}\rtimes \mathbf{Q})$ in terms of the shifted
multirectangular coordinates $\mathbf{P},\mathbf{Q}$ one can obtain a rather
straightforward expression for the spin Stanley polynomial
$\p_\pi(\mathbf{P}\rtimes \mathbf{Q})$. Applications of this result to
\emph{spin Kerov polynomials} will be discussed in a forthcoming paper
\cite{SpinAlgebraic2018}.

\subsection{Towards irreducible representations of spin groups} 

The proof of
the linear Stanley formula \eqref{eq:stanley-linear} presented in
\cite{FeraySniady2011a} was found in the following way. We attempted to
reverse-engineer the right-hand side of \eqref{eq:stanley-linear} and to find
\begin{itemize}
	\item some natural vector space $V$ with the basis indexed by combinatorial objects;
	the space $V$ should be a representation of the symmetric group $\Sym{n}$ with
	$n:=|\lambda|$, and 
	\item  a projection $\Pi\colon V\to V$ such that $\Pi$ commutes
	with the action of $\Sym{n}$ and such that its image $\Pi V$ is an irreducible
	representation of $\Sym{n}$ which corresponds to the specified Young diagram
	$\lambda$
\end{itemize}
in such a way that the corresponding character of $\Pi V$ would coincide with
the right-hand side of \eqref{eq:stanley-linear}.

Our attempt was successful: one could consider a vector space $V$ with the basis
indexed by fillings of the boxes of $\lambda$ with the numbers $[n]$. The action
of $\Sym{n}$ on this basis was given by pointwise relabelling of the values in
the boxes. The projection $\Pi$ turned out to be the Young symmetrizer with the
action given by shuffling of the boxes in the rows and columns of $\lambda$. The
resulting representation $\Pi V$ clearly coincides with the Specht module, which
concluded the proof.

\medskip

The structure of the right-hand side of \eqref{eq:spin-stanley} might be an
indication that an analogous reverse-engineering process could be applied to
the spin case. The result would be a very explicit construction of the
irreducible spin representations which would be an alternative to the somewhat
complicated approach of Nazarov \cite{Nazarov1990}.

\section{Linear characters in terms of spin characters}
	\label{sec:linear-in-terms-of-spin}

For $\lambda\in\SP$ and $\class\in\OP$ we denote
\[ \ChD_\class(\lambda):=\frac{1}{2}\ \Ch_\class\big( D(\lambda) \big).\]

The following result is an intermediate step in the proof of
\cref{thm:spin-Stanley} but it might be of independent interest. In particular,
in a forthcoming paper \cite{Matsumoto2018a} we shall discuss its applications
in the study of random strict partitions as well as random shifted standard
Young tableaux.
	
	\begin{theorem}
		\label{thm:linear-spin}
	 For any odd integers $k_1,k_2,\ldots\geq 1$ the following equalities between
functions on the set $\SP$ of strict partitions hold true:
		\begin{align}
		\label{eq:example-chd}
		\ChD_{k_1} &=  
		\p_{k_1}, \\
		\nonumber
		\ChD_{k_1,k_2} &=  \p_{k_1,k_2} +
		\p_{k_1}\ 
		\p_{k_2}, \\
		\nonumber
		\ChD_{k_1,k_2,k_3} &=  \p_{k_1,k_2,k_3} +
		\p_{k_1,k_2}\
		\p_{k_3} + 
		\p_{k_1,k_3}\
		\p_{k_2} 
		+
		\p_{k_2,k_3}\ 
		\p_{k_1},   \\
		\nonumber \vdots \\
		\label{eq:linear-in-spin}
		\ChD_{k_1,\dots,k_l} &= \sum_{\substack{ \partition: \\
				|\partition|\leq 2} } \prod_{b\in \partition} \p_{(k_i : i \in b)},
		\end{align}
		where the sum in \eqref{eq:linear-in-spin} runs over all set-partitions of
		$[l]$ into at most two blocks.
	\end{theorem}
	
	\begin{proof}
		Recall the symmetric function algebra $\Lambda=\Q[p_1,p_2,p_3,\dots]$ and its
subalgebra, \emph{the algebra of supersymmetric functions}
$\Gamma=\Q[p_1,p_3,p_5,\dots]$, where the $p_r$ are Newton's power-sums.
Define the algebra homomorphism $\varphi:\Lambda \to \Gamma$ by
		\[
		\varphi(p_r)= \begin{cases}
		2 p_r & \text{if $r$ is odd}, \\
		0 & \text{if $r$ is even}.
		\end{cases}
		\]
		Then \cite[III-8, Ex.~10]{Macdonald1995} implies that for any strict partition
$\lambda$ we have
		\begin{equation} 
		\label{eq:Schur_SchurQ}
		\varphi(s_{D(\lambda)})= 2^{-\ell(\lambda)} (Q_\lambda)^2.
		\end{equation}

		Recall the Frobenius formula for Schur functions:
		\[
		s_{\mu}= \sum_{\class} z_\class^{-1}\ \CHIL{\mu}{\class}\ p_{\class}.
		\]
		Applying the homomorphism $\varphi$ to this identity with
		$\mu=D(\lambda)$, we obtain
		\[
		\varphi(s_{D(\lambda)})= \sum_{\class \in \OP_{2n}} 2^{\ell(\class)}
		z_\class^{-1}\ \CHIL{D(\lambda)}{\class}\ p_\class.
		\]
		And, recall the Frobenius formula for Schur $Q$-functions:
		\[
		Q_\lambda = \sum_{\nu \in \OP_n} 
		2^{\ell(\nu)} z_\nu^{-1}\ \CHIS{\lambda}{\nu}\ p_\nu.
		\]
		Substituting these formulas to \eqref{eq:Schur_SchurQ},
		we have for any $\lambda\in\SP_n$
		\begin{equation} \label{eq:Schur_SchurQ2}
		\sum_{\class \in \OP_{2n}} 
		2^{\ell(\class)}
		z_\class^{-1}\ \CHIL{D(\lambda)}{\class}\ p_\class
		= 2^{-\ell(\lambda)} 
		\left( \sum_{\nu \in \OP_n} 
		2^{\ell(\nu)} z_\nu^{-1}\ \CHIS{\lambda}{\nu}\ p_\nu\right)^2.
		\end{equation}
		By comparing the coefficients of $p_{(1^{2n})}=p_{(1^{n})}p_{(1^{n})}$ in both
		sides of \eqref{eq:Schur_SchurQ2}, we find 
		\begin{equation} \label{eq:f_to_g2}
		\frac{f^{D(\lambda)}}{(2n)!}  = 
		2^{-\ell(\lambda)} \Big(\frac{g^\lambda}{n!}\Big)^2.
		\end{equation}
		
		\medskip

		First we assume that $\class$ is an odd partition
		which does not have parts equal to $1$, i.e., 
		$m_1(\class)=0$.
		By comparing the coefficients of $p_{\class \cup (1^{2n-|\class|})}$ in both
		sides of \eqref{eq:Schur_SchurQ2} we find
		\[
		\frac{\CHIL{D(\lambda)}{\class\cup (1^{n-|\class|})}}{z_{\class\cup
				(1^{n-|\class|})}} = 2^{-\ell(\lambda)} \sum_{\substack{ \mu^1, \mu^2 \\ \mu^1
				\cup \mu^2 =\class}} \frac{ \CHIS{\lambda}{\mu^1 \cup (1^{n-|\mu^1|})}}{z_{\mu^1
				\cup (1^{n-|\mu^1|})}} \frac{ 
			\CHIS{\lambda}{\mu^2 \cup (1^{n-|\mu^2|})}}{z_{\mu^1
				\cup (1^{n-|\mu^2|})}}.
		\]
		By the assumption $m_1(\class)=0$, we have
		$z_{\class \cup (1^{2n-|\class|})}= z_\class \cdot (2n-|\class|)!$
		and  $z_{\mu^i \cup (1^{n-|\mu^i|})}= z_{\mu^i} \cdot (n-|\mu^i|)!$.
		Thus, we obtain
		\[
		\frac{\CHIL{D(\lambda)}{\class\cup (1^{n-|\class|})}}{z_\class \cdot (2n-|\class|)!}
		= 2^{-\ell(\lambda)} \sum_{\substack{ \mu^1, \mu^2  \\
				\mu^1 \cup \mu^2 =\class}}
		\frac{ \CHIS{\lambda}{\mu^1 \cup (1^{n-|\mu^1|})}}{z_{\mu^1} \cdot (n-|\mu^1|)!} 
		\frac{ \CHIS{\lambda}{\mu^2 \cup (1^{n-|\mu^2|})}}{z_{\mu^2} \cdot (n-|\mu^2|)!}.
		\]
		Taking the quotient of this and \eqref{eq:f_to_g2}, 
		we have
		\begin{multline*}
		\frac{1}{z_\class} \frac{(2n)!}{(2n-|\class|)!} 
		\frac{\CHIL{D(\lambda)}{\class \cup(1^{2n-|\class|})}}{f^{D(\lambda)}}
		=\\ \sum_{\substack{ \mu^1, \mu^2  \\
				\mu^1 \cup \mu^2 =\class}}
		\frac{1}{z_{\mu^1} z_{\mu^2}} 
		\frac{n!}{(n-|\mu^1|)!} \frac{\CHIS{\lambda}{\mu^1 \cup (1^{n-|\mu^1|})}}{g^\lambda}
		\frac{n!}{(n-|\mu^2|)!} \frac{\CHIS{\lambda}{\mu^2 \cup (1^{n-|\mu^2|})}}{g^\lambda},
		\end{multline*}
		which is equivalent to 
		\[
		\Ch_\class\big(D(\lambda)\big)= 
		\sum_{\substack{ \mu^1, \mu^2  \\
				\mu^1 \cup \mu^2 =\class}}
		\frac{z_\class}{z_{\mu^1} z_{\mu^2}} 
		\p_{\mu^1}(\lambda) \p_{\mu^2}(\lambda).
		\]
		It is easy to see that
		this is equivalent to the desired formula.
		Thus, we completed the proof of the theorem under the assumption 
		$m_1(\class)=0$.
		
		\medskip
		
		In the general case we write $\pi=\tilde{\pi} \cup (1^r)$ with
$m_1(\tilde{\pi})=0$ and $r=m_1(\pi)$. We apply \cref{thm:linear-spin} for
$\tilde{\pi}$; simple manipulations with the binomial coefficients imply that
the claim holds true for $\pi$ as well. 
	\end{proof}

\section{Spin characters in terms of linear characters}
\label{sec:spin-linear}

	Formulas \eqref{eq:example-chd}--\eqref{eq:linear-in-spin} can be viewed as an upper-triangular system of
	equations with unknowns $(\p_{\class})_{\class\in\OP}$. It can be solved, for
	example
	\begin{equation}
	\label{eq:example-chd2}
	\left.
	\begin{aligned}
	\p_{k_1} &=  
	\ChD_{k_1}, \\
	\p_{k_1,k_2} &=  \ChD_{k_1,k_2} - \ChD_{k_1}\ 
	\ChD_{k_2}, \\
	\p_{k_1,k_2,k_3} &=  \ChD_{k_1,k_2,k_3} \\ & -
	\ChD_{k_1,k_2}\
	\ChD_{k_3} 
	-
	\ChD_{k_1,k_3}\
	\ChD_{k_2} 
	-
	\ChD_{k_2,k_3}\ 
	\ChD_{k_1} 
	\\ &
	+ 3 \ChD_{k_1} \ChD_{k_2} \ChD_{k_3},  \\
	\vdots
	\end{aligned} \right\}
	\end{equation}
	
The general pattern is given by the following result. %gives an explicit
In this way several problems involving spin characters are reduced to
investigation of their linear counterparts.

	\begin{theorem}
		\label{theo:spin-in-linear}
		For any $\class\in\OP$
		\begin{equation} 
		\label{eq:spin-in-linear}
		\p_\class = \sum_{\partition} (-1)^{|\partition|-1}\ (2|\partition|-3)!!\
		\prod_{b\in \partition} \ChD_{(\class_i:i \in b)}, 
		\end{equation}
		where the sum runs over all set-partitions of the set 
		$[\ell(\class)]$; by definition $(-1)!!=1$.
	\end{theorem}
	\begin{proof}
The process of solving the upper-triangular system of equations
\eqref{eq:example-chd}--\eqref{eq:linear-in-spin} can be formalized as
follows. By singling out the partition $\partition$ in
\eqref{eq:linear-in-spin} which consists of exactly one block we may express
the spin character $\p_\class$ in terms of the linear character $\ChD_\class$
and spin characters $\p_{\class'}$ which correspond to partitions
$\class'\in\OP$ with $\ell(\class')<\ell(\class)$:
		\begin{equation}
		\label{eq:spin-vs-linear3}
		\p_\class= \ChD_{\class} -
		\sum_{\substack{ \partition: \\ |\partition|= 2} } \prod_{b\in \partition} \p_{(\class_i : i \in b)}.
		\end{equation}
		By applying this procedure recursively to the spin characters on the right-hand
		side, we end up with an expression for $\p_\class$ as a linear combination (with
		integer coefficients) of the products of the form
		\begin{equation} 
		\label{eq:mysummand}
		\prod_{b\in \partition} \p_{(\class_i : i \in b)} 
		\end{equation}
		over set-partitions $\partition$ of $[\ell(\class)]$. The remaining difficulty is to
		determine the exact value of the coefficient of \eqref{eq:mysummand} in this
		linear combination.
		
		\medskip
		
		The above recursive procedure can be encoded by a tree in which each non-leaf
vertex has two children and the leaves are labelled by the factors in
\eqref{eq:mysummand} or, equivalently, by the blocks of the set-partition $I$.
Such trees are known under the name of \emph{total binary partitions}; the
cardinality of such trees with prescribed leaf labels $\partition$ is equal to
$(2 |\partition| - 3)!!$ \cite[Example 5.2.6]{Stanley1999}.
		
		Our recursive procedure involves change of the sign; such a change occurs once
		for each non-leaf vertex. Thus each total binary tree contributes with
		multiplicity $(-1)^{|\partition|-1}$ which concludes the proof.
	\end{proof}

	\section{Proof of spin Stanley formula}
	\label{sec:proof}
\begin{proof}[Proof of \cref{thm:spin-Stanley}]
	We start with \cref{theo:spin-in-linear} and substitute each normalized linear
	character $\Ch_\nu$ which contributes to the right-hand side of
	\eqref{eq:spin-in-linear} by linear Stanley character formula
	\eqref{eq:stanley-linear}. 
	
	We shall discuss in detail the case when $\pi=(\pi_1,\pi_2)$ consists of just
	two parts. We will view $\Sym{\pi_1}$, $\Sym{\pi_2}$ and $\Sym{\pi_1+\pi_2}$ as
	the groups of permutations of, respectively, the set $\{1,\dots,\pi_1\}$,
	$\{\pi_1+1,\dots,\pi_1+\pi_2\}$ and $\{1,\dots,\pi_1+\pi_2\}$. In this way we
	may identify $\Sym{\pi_1}\times\Sym{\pi_2}$ as a subgroup of
	$\Sym{\pi_1+\pi_2}$. Thanks to these notations
	\begin{multline} 
	\label{eq:two-parts}
	\p_{\pi_1,\pi_2}(\lambda) =  
{ \frac{(-1)!!}{2} \Ch_{\pi_1,\pi_2} \big( D(\lambda) \big) -
		\frac{1!! }{2^2} \Ch_{\pi_1} \big( D(\lambda) \big) 
		\ \Ch_{\pi_2} \big( D(\lambda) \big) =} \\
	\\ 
	{\frac{(-1)!!}{2} 
		\sum_{\substack{\sigma_1,\sigma_2\in\Sym{\pi_1+\pi_2} \\
				\sigma_1\sigma_2=(\pi_1,\pi_2)}}
		(-1)^{\sigma^1} N_{\sigma_1,\sigma_2}\big( D(\lambda) \big)
	}  
	-\frac{1!! }{2^2}
	\sum_{\substack{\sigma_1,\sigma_2\in\Sym{\pi_1}\times\Sym{\pi_2} 	\\ \sigma_1\sigma_2=(\pi_1,\pi_2)}}
	(-1)^{\sigma_1} N_{\sigma_1,\sigma_2}
	\big( D(\lambda) \big),
	\end{multline}
	where the last equality follows from the observation that a double sum over
factorizations of $\pi_1\in\Sym{\pi_1}$ and over factorizations of
$\pi_2\in\Sym{\pi_2}$ can be combined into a single sum over factorizations of
$(\pi_1,\pi_2)\in\Sym{\pi_1}\times \Sym{\pi_2}$.
	
	In general, 
	\begin{equation} 
	\label{eq:Stanley-not-ready}
	\p_{\pi}(\lambda) = 
	\sum_{\substack{\sigma_1,\sigma_2\in\Sym{|\pi|} \\
			\sigma_1\sigma_2=\pi}} c_{\sigma_1,\sigma_2}\ (-1)^{\sigma_1}\ N_{\sigma_1,\sigma_2}\big(D(\lambda)\big)
	\end{equation}
	for some combinatorial factor $c_{\pi_1,\pi_2}$. The exact value of this factor is equal to
	\begin{equation}	
	\label{eq:stirling}
	c_{\sigma_1,\sigma_2} = C_m= (-1) \sum_p \stirlingS{m}{p} \left(-\frac{1}{2}\right)^p (2p-3)!!, 
	\end{equation}
	where $m$ denotes the number of orbits of $\langle \sigma_1,\sigma_2\rangle$,
	and $\stirlingS{m}{p}$ denotes Stirling numbers of the second kind. 
	Indeed, the set-partition $\partition$ (over which we sum in
\eqref{eq:spin-in-linear}) can be identified with a set-partition of the set
$C(\pi)$ of the cycles of the permutation $\pi\in\Sym{|\pi|}$. With this in
mind we see that to $c_{\sigma_1,\sigma_2}$ contribute only these
set-partitions $\partition$ on the right-hand side of \eqref{eq:spin-in-linear}
for which $\partition$ is bigger than the set-partition given by the orbits of
$\langle \sigma_1,\sigma_2\rangle$. The collection of such set-partitions can
be identified with the collection of set-partitions of an $m$-element set
(i.e.~the set of orbits of $\langle \sigma_1,\sigma_2\rangle$).

The exact form of the right-hand side of \eqref{eq:stirling} is not important; the key point is that it depends only on $m$, the number of orbits of $\langle \sigma_1,\sigma_2\rangle$.
In order to evaluate its exact value $C_m$ we shall consider \eqref{eq:Stanley-not-ready} in the
special case of $\pi=1^m$. In this case $\sigma_2=\sigma_1^{-1}$; we denote by
$l=|C(\sigma_1)|$ the number of cycles of $\sigma_1$. It follows that
\[ \p_{1^m}(\lambda)=n^{\downarrow m} = 
\sum_l \stirlingF{m}{l}\ C_l\ (-1)^{m-l}\ (2n)^{l},
\]
where $n=|\lambda|$ and $\stirlingF{m}{l}$ denotes Stirling number of the first
kind. Both sides of the equality are polynomials in the variable $n$; by
comparing the leading coefficients we conclude that
\[ C_m = \frac{1}{2^m}.\]
By substituting this value to \eqref{eq:Stanley-not-ready} we conclude the proof.
\end{proof}

\section*{Acknowledgments}

We thank Valentin F\'eray and Maciej Do\l\k{e}ga for several inspiring discussions.

	\bibliographystyle{plain}
	\bibliography{Shifted.bib}

\end{document}